\documentclass[a4paper,12pt]{article}
\usepackage{amsthm}
\usepackage{amsmath}
\usepackage{latexsym}
\usepackage{amsfonts}
\usepackage{graphicx}
\usepackage{setspace}
\usepackage{framed}

\numberwithin{equation}{section}

\newcommand{\bi}{\begin{itemize}}
\newcommand{\ei}{\end{itemize}}

\newcommand{\ds}{\displaystyle}
\newcommand{\hs}{\,}
\newcommand{\hts}{\hspace*{0.1ex}}
\newcommand{\htts}{\hspace*{0.2ex}}
\newcommand{\htbs}{\hspace*{-0.1ex}}
\newcommand{\httbs}{\hspace*{-0.2ex}}

\newtheorem{lemma}{Lemma}
\newtheorem{theorem}{Theorem}
\newtheorem{conjecture}{Conjecture}
\newtheorem{corollary}{Corollary}
\newtheorem{remark}{Remark}
\newtheorem{algorithm}{Algorithm}

\newcommand{\lb}{\linebreak}

\newcommand{\fl}{\hts\mathsf{fl}}
\newcommand{\eps}{\varepsilon}
\newcommand{\cond}[1]{\mathrm{cond}(#1)}
\newcommand{\condd}[1]{\mathrm{cond}_2(#1)}

\newcommand{\mfl}[2]{#1\cdot\httbs10^{#2}}

\newcommand{\ufl}[2]{\underline{#1\cdot\htbs\htbs10^{#2}}}

\newcommand{\Matlab}{\emph{Matlab}}

\begin{document}

\title{{\Large\bf On recursive algorithms for inverting
  tridiagonal matrices} }

\author{Pawe\l\ Keller, Iwona\ Wr\'obel \\
 {\footnotesize Warsaw University of Technology, Faculty of
  Mathematics and Information Science} }

\date{}


\maketitle

\begin{abstract}
If $A$ is a tridiagonal matrix, then the equations $AX=I$ and $X\httbs A=I$
defining the inverse $X$ of $A$ are in fact the second order recurrence
relations for the elements in each row and column of $X$. Thus, the
recursive algorithms should be a natural and commonly used way for inverting
tridiagonal matrices -- but they are not. Even though a variety of such
algorithms were proposed so far, none of them can be applied to numerically
invert an arbitrary tridiagonal matrix. Moreover, some of the methods suffer a huge
instability problem. In this paper, we investigate these problems very
thoroughly. We locate and explain the different reasons the recursive
algorithms for inverting such matrices fail to deliver satisfactory
(or any) result, and then propose new formulae for the elements
of $X=A^{-1}$ that allow to construct the asymptotically fastest possible
algorithm for computing the inverse of an arbitrary tridiagonal matrix $A$,
for which both residual errors, $\|AX-I\|$ and $\|X\httbs A-I\|$, are always very small.
\end{abstract}

%
% ...

\section{Introduction} \label{SecIntro}
\setcounter{equation}{0}

Matrix inversion has quite numerous applications in statistics,
cryptography, computer graphics, etc. It is hard to imagine a
computation system or a scientific programming environment without
a library or a function that calculates the inverse of a matrix.

In many applications, where the inverse of a matrix appears, there
is no actual need for direct computation of the inverse, as the corresponding
problem can be solved by computing the solution of a matrix equation or,
in particular cases, a system of linear equations. There are
some problems, however, where the inverse of a matrix
is indeed required.

Even though the problem of computing matrix inverse has been extensively
studied and described in many numerical monographs, there is a
variety of new papers dealing with the subject. In the last
couple of years, a number of algorithms for inverting structured
matrices (block, banded) were proposed. This paper focuses on recursive
algorithms for inverting tridiagonal matrices.

The existing recursive algorithms for computing the inverse of this kind
of matrices are not very popular and are not commonly used. This is mostly
because there are wide classes of matrices for which these algorithms cannot
be applied due to their instability or other kind of limitations (it will be
described in more detail throughout the paper). In this paper, we analyse
the reasons of these disadvantages and search for the way to eliminate\lb them.
As a result, new formulae for the elements of the inverse of a tridiagonal
matrix $A$ are proposed, which allow a very fast and accurate computation
of $X=A^{-1}$, and can be applied to any non-singular tridiagonal matrix.
To our knowledge, it is the only method that guarantees, in general,
very small, both left and right, residual errors.

Let\hs $A \in\mathbb{R}^{n\times n}$\hs be a tridiagonal matrix,
i.e.\ a matrix whose elements satisfy $a_{ij} = 0$ if $|i-j|>1$
($i,j = 1,2,\dots,n$). The most common algorithms for evaluating $A^{-1}$
are based on solving the matrix equation $AX = I$ (or $X\httbs A=I$; throughout
the paper, by $X$, we shall denote the numerically computed inverse of a tridiagonal
matrix $A$, or the inverse which is to be computed, and by $I$, the identity matrix).
If the above equation is solved using, for example, Gaussian elimination with partial pivoting
or using orthogonal transformations (e.g., Givens rotations), the
computed inverse satisfies (c.f.\ \cite[\S\hts 6.12]{Datta10Book})
\begin{equation*}
  \| AX - I \htts\|\hs \leq\hs \eps\hts K\hts\cond{A} ,
\end{equation*}
where $\eps$ is the machine precision, $K$ is a small constant,
and $\mathrm{cond}(A)$ is the condition number of $A$. In such a case,
it is easy to verify that the relative error $\|A^{-1}\|^{-1}\|X-A^{-1}\|$
satisfies a similar bound:
\begin{equation*}
  \frac{\| X - A^{-1}\|}{\|A^{-1}\|} \leq\hs \| AX - I \htts\|\hs \leq\hs \eps\hts K\hts\cond{A} .
\end{equation*}
However, the bound is not so favourable in the case of the second residual error.
We have, as $X = A^{-1}AX = A^{-1}\big(I + AX - I\big)$,
\begin{align*}
  \| X\httbs A - I \htts\|
    & =\hs \| A^{-1}\big(I + AX - I\big) A - I\hts\| \\[0.5ex]
    & =\hs \| A^{-1}\big(AX - I\big) A \|\hs \leq\hs
           \cond{A} \| AX - I\hts\| \hs \leq\hs \eps\hts K\hts\cond{A}^2 .
\end{align*}
Indeed, consider the following tridiagonal matrix $A\in\mathbb{R}^{10\times10}$
whose elements $a_{ij}$ for $i,j \in\{1,2,\dots,10\htts\}$, $|i-j|\leq 1$ are listed column-wise
in the following table:
\begin{equation*}
\begin{bmatrix}
          &    0   &  1/83  &  1/70  &  1/65  &  1/49  &   16   &   49   &   57   &  70\hs \\
       1  &  1/98  &  1/84  &  1/53  &   92   &   55   &   86   &  1/84  &  1/49  &  83\hs \\
\htts 79  &   61   &   18   &    3   &  1/32  &  1/37  &  1/45  &    0   &    0   &
\end{bmatrix}\hts.
\end{equation*}
The condition number $\condd{A} \simeq \mfl{6.1}{8}$.
If the inverse $X$ is computed in the double precision arithmetic
($\eps=2^{-52}\simeq\mfl{2.2}{-16}$) using the \Matlab\ command \texttt{X\,=\,\hs A\textbackslash eye(10)},
i.e.\ using the Gaussian elimination with partial pivoting, then we will obtain the result
that satisfies
\begin{equation*}
 \| AX - I\hts\|_2\hs \simeq\hs \mfl{5.8}{-9}\hs \simeq\hs \eps\hts 0.043\hs\condd{A} .
\end{equation*}
For the left residual error, however, we have
\begin{equation*}
 \| X\httbs A - I\hts\|_2\hs \simeq\hs \textbf{3.\hts4}\hs \simeq\hs \mfl{5.8}{8}\htts \| AX - I\hts\|_2\hs \simeq\hs
 0.95\htts \condd{A} \| AX - I\hts\|_2,
\end{equation*}
and there are elements larger than $3$ outside the main diagonal in the matrix $X\httbs A$.
Such an inverse may not always be considered as a very satisfactory one. This problem has been already
considered by Higham in \cite{DuCroz92} and \cite[\S\hts 13.3]{Higham02Book}, who proposed a symmetric
algorithm for computing $A^{-1}$ based on $LU$ factorisation. However, as the algorithm
uses no pivoting strategy, it is not stable in general.

The asymptotic complexity of the most common algorithms for inverting
tridiagonal matrices equals: $3n^2+O(n)$ in the case of Gaussian elimination
without pivoting, $5n^2+O(n)$ in the case of Gaussian elimination with partial
pivoting, and $7.5n^2+O(n)$ in the case of the algorithm based on Givens rotations.
In the present paper, we will propose some new formulae which allow to construct
a stable, always applicable algorithm with the smallest possible asymptotic complexity:\
$n^2+O(n)$. The authors do not know the formal proof of stability of the algorithm yet,
but strong justification is given to support the conjecture that the inverse $X$
of a tridiagonal matrix $A\in\mathbb{R}^{n\times n}$ computed
using the proposed method satisfies
\begin{equation*}
  \max\big\{ \| AX - I\hts\|, \| X\httbs A - I\hts\| \big\}\hs \leq\hs \eps\hts K(n)\hts \cond{A} ,
\end{equation*}
where $K(n) = O(n)$.

In the incoming section, we shall present a short review of the recursive approach
to the problem of inverting a tridiagonal matrix. We shall also recall the basic
facts from the theory of difference equations that will help to explain the reasons
several recently proposed recursive algorithms for inverting tridiagonal matrices
fail (or are unstable) for some important classes of matrices.

In Section \ref{SecNewAlg}, we present a new efficient method, also based on some
recursions, which can be successfully applied to invert any non-singular tridiagonal matrix.

\section{A short review of the recursive algorithms and the theory
of the second order difference equations} \label{SecRec}
\setcounter{equation}{0}

In this section, we review several algorithms for recursive computation of
the inverse of a tridiagonal matrix. Showing the strong and the weak sides of
these algorithms will lead us to the main result of the paper.

\subsection{The naive recursion} \label{SubSecNaive}

It is well known that for\hs $B \in\mathbb{R}^{n\times n}$\hs and\,
$c \in\mathbb{R}^n$\hs the product\hs $Bc$ can be interpreted as the
linear combination of columns of the matrix $B$,
\begin{equation*}
  Bc = \sum\limits_{k=1}^n c_k B_k\,,
\end{equation*}
where $B_k$ denotes the $k$-th column of $B$. Such interpretation is very
convenient if $c$ has only a few non-zero elements.

Let $A \in\mathbb{R}^{n\times n}$\hs be a tridiagonal matrix
($a_{ij} = 0$ if $|i-j|>1$). Then, the equation $X\httbs A = I$
implies the following relations:
\begin{align*}
  a_{k-1,k} X_{k-1} + a_{k,k} X_k +
  a_{k+1,k} X_{k+1} & = I_k \quad (1 < k < n)\,, \\[0.875ex]
  a_{n-1,n} X_{n-1} + a_{n,n} X_n & = I_n\,.
\end{align*}
Thus, if we know the last column $X_n$ of $X$, then we can easily
recursively compute the whole inverse matrix
$X$:\vspace{-0.25ex}
\begin{equation}
\renewcommand{\arraystretch}{2}
\left.\begin{array}{l}\vspace{-6ex}\\
  \ds{X_{n-1}}\hs  =\hs\hs I_n\hs -\hs
    \frac{a_{n,n}}{a_{n-1,n}} X_n\hs,\\
  \hts\ds{X_{k-1}}\hs  =\hs\hs I_k\hs -\hs
    \frac{a_{k,k}}{a_{k-1,k}} X_k\hs -\hs
    \frac{a_{k+1,k}}{a_{k-1,k}} X_{k+1}
  \quad
    (n > k > 1)\hs,
\end{array}\quad\right\}
\label{SimplRecAlg}
\end{equation}
where we assume that\hs $a_{k-1,k}\neq 0$ ($1 < k\leq n$).

The above simple observation laid the basis of two recent algorithms,
\cite{Mikkawy08} and \cite{Hadj08}, for inverting tridiagonal matrices.
In \cite{Mikkawy08}, the last column $X_n$ is computed using the LU
factorisation without pivoting, while in \cite{Hadj08}, the Miller
algorithm --- a classical algorithm for computing the \emph{minimal}
solutions of second order difference equations --- has
been rediscovered.

Let us test the stability of the above scheme for
a very small and very well conditioned matrix\vspace{0.25ex}
\begin{equation}
  A\hs\hs =\hs\hs
  \begin{pmatrix}\hspace*{0.5ex}
    2016 & 1 & & & & \\\hspace*{0.5ex}
    1 & 2016 & 1 & & & \\
    & 1 & 2016 & 1 & & \\
    & & 1 & 2016 & 1 & \\
    & & & 1 & 2016 & 1\hspace*{1ex} \\
    & & & & 1 & 2016\hspace*{1ex} \\
  \end{pmatrix}.
\label{Matrix2016}
\vspace{0.5ex}
\end{equation}
For the inverse matrix $X$ computed in the double precision
arithmetic using the algorithm \cite{Mikkawy08} or \cite{Hadj08}
(the result does not depend on which of the above ways the last column
of $X$ is calculated), we have
\begin{equation*}
  \| AX - I\hts\|_2\hs \simeq\hs \| X\httbs A - I\hts\|_2\hs \simeq\hs \textbf{4.4} ,
\end{equation*}
while $\condd{A}=1.002$. If the inverse is evaluated in \Matlab:
\texttt{X\,=\,\htts A\textbackslash eye(6)}, the residual errors satisfy
$\|AX-I\hts\|_2\simeq\|X\httbs A-I\hts\|_2\simeq\eps\simeq\mfl{2.2}{-16}$.
The explanation of such a huge instability of the simple recursive algorithms
based on (\ref{SimplRecAlg}) is quite simple if we recall some basic facts
of the theory of linear second order difference equations
(see, for example, \cite{Gautschi67} or \cite{Wimp84Book}).

The three term (second order) homogenous recurrence (difference) equation
can be written in general form as follows:
\begin{equation}
  a_{-\htbs1}(k)\hts x_{k-1}\hs +\hs a_{0}(k)\hts x_{k}\hs +\hs
  a_{1}(k)\hts x_{k+1}\hs =\hs 0
    \qquad ( k > 1,\hs\hs a_{\pm1}(k)\neq 0) ,
\label{GenDiffEq}
\end{equation}
where $a_{j}(k)$ ($j=-1,0,1$) are known coefficients, and $\{x_k\}_{k=1}^{\infty}$
is a solution we are looking for. The equation (\ref{GenDiffEq}) has two-dimensional
space of solutions. If there exist two linearly independent solutions $\{u_k\}$
and $\{y_k\}$ such that
\begin{equation*}
  \lim\limits_{k\to\infty}\frac{u_k}{y_k}\hs =\hs 0\hs,
\end{equation*}
then $\{u_k\}$ is called a \textit{minimal}, and $\{y_k\}$ is called
a \textit{dominant} \hts solution. The \textit{forward recursion} algorithm,
\begin{equation*}
  x_{k+1} =
   - a_{1}(k)^{\htbs-1}\big(a_{0}(k)\hs x_{k}
   + a_{-\htbs1}(k)\hs x_{k-1}\big)\hts,
\end{equation*}
is stable for dominant solutions only, while the
\textit{backward recursion},
\begin{equation*}
  x_{k-1} =
  -a_{-\htbs1}(k)^{\htbs-1}\big(a_{0}(k)\hs x_{k}
  +a_{1}(k)\hs x_{k+1} \big)\hts,
\end{equation*}
is stable only for minimal solutions (at the present point, we do not consider
the problem of obtaining the initial values for the above recurrences).
If \hts for any pair\hs $\{u_k\}$,\hs $\{y_k\}$\hs of independent
solutions we have
\begin{equation*}
  \liminf\limits_{k\to\infty}\frac{u_k}{y_k}\hs =\hs w > 0\hs,\qquad
  \limsup\limits_{k\to\infty}\frac{u_k}{y_k}\hs =\hs \mathcal{W} < \infty\hs,
\end{equation*}
where $w,\hs \mathcal{W}^{-1}\gg \eps$, then both forward and
backward recursion algorithms are stable (in the asymptotic sense).

In practice, we are usually interested in computing only a part
of a solution of the recurrence equation (\ref{GenDiffEq}), i.e.\ the values of
$x_k$ for $1\leq k\leq n$ for some $n>0$. Note that the starting point ($k=1$)
and the main (forward) direction of the recursion ($k=2,3,\dots$)
is only a convention.

The space of all minimal solutions of the linear second order difference equation
is one-dimensional. An important property of the minimal solutions and
the algorithm for computing the values of such a solution is given
in the following theorem.

\begin{theorem}[Miller] \label{TwMiller}
Assume that a three term recurrence equation (\ref{GenDiffEq})
has a minimal solution $\{u_k\}$ which satisfies a normalising condition
\begin{equation*}
  \sum_{k=1}^{M} c_k u_k = S \neq 0  \qquad (S,c_k\in\mathbb{R},\,\,\, 0 < M < \infty) .
\end{equation*}
For $\hts n>1\hts$, define the values $x^{[n]}_k$ ($1\leq k\leq n+1$) as follows:
\begin{align*}
  & x^{[n]}_{n+1} = 0, \quad x^{[n]}_n = d \neq 0,\\[0.75ex]
  & x^{[n]}_{k-1} =
  - a_{-\htbs1}(k)^{\htbs-1}\big(a_{0}(k)\hs x^{[n]}_k
  + a_{1}(k)\hs x^{[n]}_{k+1}\big) \qquad (n \geq k > 1)\hts.
\vspace{0.75ex}
\end{align*}
Then, for each $k\geq 1$
\begin{equation*}
  \quad\lim\limits_{n\to\infty}
    S_n^{-1} S x^{[n]}_k =\hs
    \htts u_k\hts,
\end{equation*}
where\vspace{-1ex}
\begin{equation*}
  S_n = \sum\limits_{k=0}^{M}c_k x^{[n]}_k
\vspace{-0.5ex}
\end{equation*}
(we assume that $x^{[n]}_k=0$ if $k>n$).
\end{theorem}
\begin{proof}
See \cite{Miller52} or \cite[\S\hts 4]{Wimp84Book}.
\end{proof}

Let us consider again the recursion (\ref{SimplRecAlg}), but only for the
elements of the first row of the inverse matrix $X$. We may of course write that $x_{1,n+1}=0$.
Additionally, if $a_{k-1,k}\neq 0$ for all $1 < k\leq n$, then, theoretically,  $x_{1,n}\neq 0$ (this
will be justified in the later part of the paper; see also \cite{Lewis82}). Consequently, we have\vspace{-1ex}
\begin{equation}
\renewcommand{\arraystretch}{1.6}
\begin{array}{l}
     x_{1,n+1}\hts =\hts 0, \quad x_{1,n}\hts \neq\hts 0, \\
\hts x_{1,k-1}\hts =\hts
     -{a_{k-1,k}^{-1}}{\big(a_{k,k}\hts x_{1,k}
     + a_{k+1,k}\hts x_{1,k+1}\big)}
       \qquad (n \geq k > 1)
\end{array}
\label{Row1SimplRec}
\end{equation}
and (also from the equation $X\httbs A=I$\hts) $a_{1,1}x_{1,1}+a_{2,1}x_{1,2} = 1 \neq 0$.
Comparing the above formulae to the ones of Theorem \ref{TwMiller}, we may conclude that
if the recurrence equation for the first row of the inverse matrix $X$ has minimal
and dominant solutions, then the elements $x_{1,1},x_{1,2},\dots,x_{1,n}$ of the
first row of $X$ behave more like a minimal solution than like a dominant one.
Therefore, the recursion (\ref{SimplRecAlg}) which is the backward
recursion for $x_{1,k}$ ($n\geq k\geq 1$) is stable for all
elements of the first row of the matrix $X$. In the case of the $s$-th row,
we have\vspace{-1ex}
\begin{equation*}
\renewcommand{\arraystretch}{1.6}
\begin{array}{l}
     x_{s,n+1}\hts =\hts 0\hts, \qquad
     x_{s,k-1}\hts =\hts
     -{a_{k-1,k}^{-1}}{\big(a_{k,k}\hts x_{s,k}
     + a_{k+1,k}\hts x_{s,k+1}\big)}
       \qquad (n\geq k > s) .
\end{array}
\end{equation*}
Note that this recurrence is exactly the same as the recurrence (\ref{Row1SimplRec}), only the
starting value $x_{1,n}$ is replaced by $x_{s,n}$. This implies that the recursion (\ref{SimplRecAlg})
is stable for all elements in the upper triangle of $X$. In the lower triangle of $X$,
the situation is reversed. The recursion (\ref{SimplRecAlg}) is the forward
recursion algorithm for the elements $x_{s,k}$ ($s > k \geq 1$, $s>1$), and
therefore may be unstable if the corresponding recurrence equation,
\begin{equation*}
  a_{k+1,k}\hts x_{s,k+1} + a_{k,k}\hts x_{s,k} + a_{k-1,k}\hts x_{s,k-1} = 0 \qquad (s > k\geq 2) ,
\end{equation*}
has minimal and dominant solutions. Indeed, in the case of the matrix $X$ computed
using (\ref{SimplRecAlg}), where $A$ is given by (\ref{Matrix2016}), we have
\begin{equation*}
  | X - A^{-1} |\hs\hs \simeq\hs\hs \|A^{-1}\|
  \begin{pmatrix}
  \htts\ufl{6}{-16} & \mfl{2}{-19} & \mfl{1}{-22} & \mfl{1}{-25} & \mfl{6}{-29} & \mfl{3}{-32}\hs \\
  \htts\mfl{6}{-13} & \ufl{4}{-16} & \mfl{3}{-19} & \mfl{1}{-22} & \mfl{1}{-25} & \mfl{6}{-29}\hs \\
  \htts\mfl{1}{-09} & \mfl{8}{-13} & \ufl{6}{-16} & \mfl{3}{-19} & \mfl{2}{-22} & \mfl{1}{-25}\hs \\
  \htts\mfl{1}{-06} & \mfl{8}{-10} & \mfl{4}{-13} & \ufl{4}{-16} & \mfl{2}{-19} & \mfl{1}{-22}\hs \\
  \htts\mfl{2}{-03} & \mfl{1}{-06} & \mfl{4}{-10} & \mfl{2}{-13} & \ufl{2}{-16} & \mfl{2}{-19}\hs \\
  \htts\mfl{4}{+00} & \mfl{2}{-03} & \mfl{1}{-06} & \mfl{5}{-10} & \mfl{2}{-13} & \ufl{2}{-16}\hs
\end{pmatrix}
\end{equation*}
(for readability, diagonal elements are underlined). 

The recursion (\ref{SimplRecAlg}) was obtained from the matrix equation $X\httbs A=I$. The second, twin
equation, $AX=I$, implies the analogous recurrence for rows of the matrix $X$.
From the discussion above, the important result follows.

\begin{corollary} \label{WnTowardsDiag}
If the elements of the inverse $X$ of a tridiagonal matrix are computed recursively,
then, in general, the algorithm is stable only if the recurrences are carried out towards
the main diagonal of $X$.
\end{corollary}

We end this subsection by formulating the algorithm for stable recursive computation
of the last column of $X$ (the algorithm is a consequence of Theorem \ref{TwMiller}
and the equation $AX_n=I_n$, and is a particular case of the
\emph{Miller backward recursion algorithm}):
\begin{align*}
    y_{0} & =\hts 0, \quad y_{1}\hts =\hts 1 \\[0.75ex]
  y_{k+1} & =\hts - a_{k,k+1}^{-1}\big(a_{k,k-1}\hts y_{k-1} + a_{k,k}\hts y_{k}\big)
             \quad\, (1\leq k < n), \\[0.75ex]
  x_{k,n} & =\hts f y_{k}
             \quad\, (1 \leq k \leq n),
\end{align*}
where the normalising factor $f = \big( a_{n,n-1} y_{n-1} + a_{n,n} y_{n} \big)^{-1}$.

\subsection{Two-way recursion} \label{SubSecTwoWay}

By Corollary \ref{WnTowardsDiag}, we conclude that a single
recursion, like (\ref{SimplRecAlg}), cannot lead to a stable algorithm
for inverting tridiagonal matrices. However, as the algorithm (\ref{SimplRecAlg})
computes the upper triangle of the inverse matrix $X$ correctly, we may use
a similar scheme to compute the elements in the lower triangle, namely,
compute the first column $X_1$ using the Miller algorithm, and then
compute recursively the elements below the diagonal in the consecutive
columns $X_k$ ($k=2,3,\dots,n-1$). The scheme looks as follows:
compute the columns $X_1$ and $X_n$ using the Miller
algorithm, and then set\vspace{-1ex}
\begin{equation}
\renewcommand{\arraystretch}{2.25}
\hspace*{-3ex}\left.\begin{array}{rl}\vspace*{-7ex}\\
  \ds\big\{x_{s,n-1}\big\}_{s\hts=\hts1}^{n\htbs-\htbs1} \hspace*{-0.875ex} & \htbs=\htts\ds
    -\hs\frac{a_{n,n}}{a_{n-\htbs1\htbs,n}}\hs
     \big\{x_{s,n}\big\}_{s\hts=\hts1}^{\hts n\htbs-\htbs1}\hts, \\
  \ds\big\{x_{s,k-1}\big\}_{s\hts=\hts1}^{\hts k\htbs-\htbs1} \hspace*{-0.875ex} & \htbs=\htts\ds
    -\hs\frac{a_{k,k}}{a_{k-\htbs1\htbs,k}}\hs
      \big\{x_{s,k}\big\}_{s\hts=\hts1}^{\hts k\htbs-\htbs1}\htts
    -\hs\frac{a_{k+1,k}}{a_{k-\htbs1\htbs,k}}\hs
      \big\{x_{s,k+1}\big\}_{s\hts=\hts1}^{\hts k\htbs-\htbs1} \quad
  (n > k > 2)\hts, \\
  \ds\big\{x_{s,2}\big\}_{s\hts=\hts3}^n \hspace*{-0.875ex} & \htbs=\htts\ds
    -\hs\frac{a_{1,1}}{a_{2,1}}\hs
      \big\{x_{s,1}\big\}_{s\hts=\hts3}^n\hts, \\
  \ds\big\{x_{s,k+1}\big\}_{s\hts=\hts k\htbs+\htbs2}^n \hspace*{-0.875ex} & \htbs=\htts\ds
    -\hs\frac{a_{k,k}}{a_{k+\htbs1\htbs,k}}\hs
      \big\{x_{s,k}\big\}_{s\hts=\hts k\htbs+\htbs2}^n\htts
    -\hs\frac{a_{k-1,k}}{a_{k+\htbs1\htbs,k}}\hs
      \big\{x_{s,k-1}\big\}_{s\hts=\hts k\htbs+\htbs2}^n \quad
  (1 < k < n-1)\hts.
\end{array}\right\}
\label{TwoWayRecAlg}
\vspace{0.25ex}
\end{equation}
We assume that $a_{k-1,k}\neq 0$ and $a_{k,k-1}\neq 0$ for $1 < k\leq n$. The above
algorithm appears to be presented in print for the first time, but it is still not a very good
one as we shall justify below.

Observe that the formulae (\ref{TwoWayRecAlg}) are based on the equations
\begin{equation}
  a_{k-1,k}\hts x_{s,k-1} + a_{k,k}\hts x_{s,k} + a_{k+1,k}\hts x_{s,k+1} = 0
    \quad\,\, (1 \leq s \leq n,\,\, 1 \leq k \leq n,\,\, k\neq s )
\label{XAscalarRec}
\end{equation}
(throughout the paper we assume than $a_{ij}=0$ and $x_{ij}=0$ if $i<1$, or $i>n$, or $j<1$, or $j>n$),
while the Miller algorithm for computing the columns $X_1$ and $X_n$ is based on the equations\vspace{-0.25ex}
\begin{equation}
  a_{k,k+1}\hts x_{k+1,j} + a_{k,k}\hts x_{k,j} + a_{k,k-1}\hts x_{k-1,j} = 0
    \quad\,\, (1 \leq j \leq n,\,\, 1 \leq k \leq n,\,\, k\neq j ) .
\label{AXscalarRec}
\end{equation}
Theoretically, in our applications of the Miller algorithm we use the above equations only for
$j\in\{1,n\}$. However, as $k\neq j$, all the equations (\ref{AXscalarRec}) for $1\leq j < n$ and $k>j$
are in fact the same recurrence equations (an analogous observation is true for $n\geq j > 1$ and $k<j$).
The normalising factors used for computing the columns $X_1$ and $X_n$ by the Miller algorithm
are derived from the two additional equations,
\begin{equation*}
  a_{1,1}\hts x_{1,1} + a_{1,2}\hts x_{2,1} = 1 \quad\mathrm{and}\quad
  a_{n,n-1}\hts x_{n-1,n} + a_{n,n}\hts x_{n,n} = 1 .
\end{equation*}
But the remaining equations,\vspace{-1.25ex}
\begin{equation}
\renewcommand{\arraystretch}{1.33}
\left.\begin{array}{l}
  a_{k-1,k} x_{k,k-1} + a_{k,k} x_{k,k} + a_{k+1,k} x_{k,k+1} = 1 , \\
  a_{k,k+1} x_{k+1,k} + a_{k,k} x_{k,k} + a_{k,k-1} x_{k-1,k} = 1 ,
\end{array}\quad\right\}
\label{MissingOnes}
\end{equation}
for $1 < k < n$, are nowhere used in the above algorithm. In other words,
the presented two-way recursion algorithm computes the matrix $X$ which
satisfy the following system of matrix equations
\begin{equation*}
  X\htbs A = D, \qquad  A\hts X = H,
\end{equation*}
where $D$ and $H$ are diagonal matrices for which we only
know that $h_{1,1} = h_{n,n} = 1$. An interesting question
arises: does the algorithm based on (\ref{TwoWayRecAlg}) computes the actual inverse of
a tridiagonal matrix $A$, assuming that the computations are exact? The answer
to the above question is delivered by the following theorem.
\begin{theorem} \label{TwTwoOnesOnly}
Let $A\in\mathbb{R}^{n\times n}$ be a non-singular tridiagonal matrix
such that $a_{k,k-1}\neq 0$ and $a_{k-1,k}\neq 0$ for $1 < k\leq n$.
If a matrix $X$ is a solution of the system of matrix equations
\begin{equation}
  X\htbs A = D, \qquad  A\hts X = H,
\label{DHeqs}
\end{equation}
where $D$ and $H$ are diagonal matrices and $H$ satisfy
$h_{1,1} = h_{n,n} = 1$, then $D=H=I$ and, consequently, $X=A^{-1}$.
\end{theorem}
\begin{proof}
The equations (\ref{DHeqs}) imply that
\begin{equation*}
AD = A\hts X\htbs A = HA.
\end{equation*}
Equating the corresponding elements of the matrices $AD$ and $HA$ in a proper order, and
recalling that $h_{1,1}=1$ and $a_{i,j}\neq 0$ for $|i-j|=1$, we obtain, in sequence,
\begin{equation}
  1 \,=\, h_{1,1} \,=\, d_{2,2} \,=\, h_{3,3} \,=\, d_{4,4} \,=\, h_{5,5} \,=\, \dots\, .
\label{dh1forward}
\end{equation}
Similarly, but using the condition $h_{n,n}=1$, we get
\begin{equation}
  1 \,=\, h_{n,n} \,=\, d_{n-1,n-1} \,=\, h_{n-2,n-2} \,=\,
                        d_{n-3,n-3} \,=\, h_{n-4,n-4} \,=\, \dots\, .
\label{dh1backward}
\end{equation}
If $n$ is even, then from (\ref{dh1forward}) and (\ref{dh1backward}) we
immediately obtain that $d_{i,i}=h_{i,i}=1$ for $1\leq i\leq n$. If $n$ is odd,
then at least one of the diagonal elements of $A$ is different from 0 (otherwise
there would be $\det(A)=0$). For simplicity, assume that $a_{1,1}\neq 0$. Then,
\begin{equation*}
  1 \,=\, h_{1,1} \,=\, d_{1,1} \,=\, h_{2,2} \,=\, d_{3,3} \,=\, h_{4,4} \,=\, \dots\, .
\end{equation*}
By combining the above result with (\ref{dh1forward}), again, we get $d_{i,i}=h_{i,i}=1$ for $1\leq i\leq n$.
\end{proof}

The above theorem implies that in theory, the algorithm based on the recursions (\ref{TwoWayRecAlg})
computes the correct inverse $X$ of a tridiagonal matrix $A$ with non-zero elements on the sub-
and super- diagonals. The situation is a little different in practice. The numerical performance
of this algorithm is far from perfection. As the lower and the upper triangles of $X$ are computed
completely independently, the values of magnitude close to $\eps\htts\cond{A}^2$ may appear%
\footnote{See Section \ref{SecNewAlg} for more detailed explanation.}
along the main diagonals of the residual matrices $|AX-I|$ and $|X\httbs A-I|$. Thus, the
search for a better algorithm has to be continued. Before it is done, we shall consider for
a moment the complexity of the methods presented so far.

It is readily seen that the algorithm based on (\ref{TwoWayRecAlg}), like the
one based on (\ref{SimplRecAlg}), has the complexity equal to $3n^2+O(n)$. However,
if we take a closer look on the equations in (\ref{TwoWayRecAlg}), we can see that the
recurrence for the elements in the upper triangle of $X$ (the first two lines) is exactly
the same for each value of the row index $s$, only the initial values are different.
Thus, we may carry out the recurrence only once, and then scale the result
according to the corresponding initial values:\vspace{-0.75ex}
\begin{equation}
\renewcommand{\arraystretch}{2.125}
\left.\begin{array}{rl}\vspace*{-6.5ex}\\
  \ds \hat{z}_{n-1} \hspace*{-0.75ex} & =\hs\ds
    -\hs\frac{a_{n,n}}{a_{n-\htbs1\htbs,n}}\hts
      \hts 1\hts, \\
  \ds \hat{z}_{k-1} \hspace*{-0.75ex} & =\hs\ds
    -\hs\frac{a_{k,k}}{a_{k-\htbs1\htbs,k}}\hts
      \hat{z}_k\hs
    +\hs\frac{a_{k+1,k}}{a_{k-\htbs1\htbs,k}}\hts
      \hat{z}_{k+1}
    \qquad (n > k > 2)\hts , \\
  \ds\big\{x_{s,k-1}\big\}_{s\hs=\hs1}^{k-\htbs1} \hspace*{-0.75ex} & =\hs
      \hat{z}_{k-1}\big\{x_{s,n}\big\}_{s\hs=\hs1}^{k-\htbs1}
   \qquad (n \geq k > 2)\hts.
\end{array}\quad\right\}
\label{FastTwoWayRec}
\end{equation}
The above scheme uses $\frac{1}{2}n^2+O(n)$ arithmetic operations. If a similar modification
is done for the computation of the lower triangle of $X$, the complexity of the whole algorithm
will drop to the smallest possible --- as the inverse of a tridiagonal matrix has, in general,
$n^2$ different elements --- asymptotic value: $n^2+O(n)$.

\subsection{The Lewis algorithm} \label{SubSecLewis}

In order to improve the numerical properties of the recursive algorithm for
inverting tridiagonal matrices, we should include the known dependence between
the elements in the upper triangle of the inverse matrix $X$ and the elements
in the lower triangle. The simple formula that relates these elements
was given in \cite{Lewis82}.
\begin{lemma}[Lewis] \label{LmLewisSymm}
If $A\in\mathbb{R}^{n\times n}$ is
a tridiagonal matrix such that $a_{k,k+1}\neq 0$ for
$k=1,2,\dots,n\htbs-\htbs1$, and $X = A^{-1}$, then
\begin{equation}
  \hspace*{-5ex}x_{k+j,k}\hs = \hs\Bigg(
    \prod_{i=k}^{k+j-1}\frac{a_{i+1,i}}{a_{i,i+1}}\Bigg)x_{k,k+j}
    \qquad(1\leq k < n,\,\, 1\leq j\leq n-k)\hts.
\label{LewisSymm}
\end{equation}
\end{lemma}
\begin{proof}
See \cite{Lewis82}.
\end{proof}

The formula (\ref{LewisSymm}) is a direct consequence (cf.\ \cite{Lewis82}) of the equations
\begin{equation}
  \sum_{i=-1}^{1} a_{k+i,k}\hts x_{k,k+i}\hs =\hs
  \sum_{i=-1}^{1} a_{k,k-i}\hts x_{k-i,k}\hs =\hs 1 \qquad (1\leq k\leq n) ,
\label{NotMissingOnes}
\end{equation}
i.e.\ is implied, in particular, by the equations (\ref{MissingOnes}) that were
not used in the algorithm presented in the previous subsection. Now, the following
algorithm with a very favourable numerical properties can be formulated:
compute the last column $X_n$ using the Miller algorithm; compute the remaining
elements in the upper triangle of $X$ using (\ref{FastTwoWayRec}) for $k > 1$,
instead of for $k > 2$; compute the lower triangle of $X$ using (\ref{LewisSymm}).
This algorithm works very well if the two following conditions are satisfied:
if $a_{k-1,k}\neq 0$ for all $2\leq k\leq n$, and if the values $\hat{z}_k$ ($n > k\geq 1$)
in (\ref{FastTwoWayRec}) do not grow too large, causing floating-point overflow
(which is, unfortunately, a quite frequent case). A very similar algorithm was
formulated in \cite{Lewis82}, where only the last column $X_n$ is computed
in a slightly different (but mathematically and numerically equivalent)
way.

Note that due to (\ref{LewisSymm}), the complexity of the
above scheme grew to $\frac{3}{2}n^2 + O(n)$. The complexity,
however, can be reduced back to $n^2 + O(n)$ if we apply the
relation (\ref{LewisSymm}) in a little different way.
The following theorem is one of the main
results of \cite{Lewis82}:
\begin{theorem}[Lewis] \label{TwLewis}
Assume that $A$ is a non-singular tridiagonal matrix which satisfies
$a_{k-1,k}\neq 0$ and $a_{k,k-1}\neq 0$ for $2\leq k\leq n$. Let the
sequences $\{\hat{z}_k\}$, $\{z_k\}$, $\{e_k\}$
be defined in the following way:\vspace{-1.5ex}
\begin{equation}
\renewcommand{\arraystretch}{2.25}
\hspace*{-0.5ex}\left.\begin{array}{l}\vspace*{-7.25ex}\\
  \ds \hat{z}_n \hts =\hs
      \hts1,\httbs\quad
  \ds \hat{z}_{n-1}\hts =\hs
    -\hs\frac{a_{n,n}}{a_{n-\htbs1\htbs,n}},\httbs\quad
  \ds \hat{z}_{k-1}\hts =\hs
    -\hs\frac{a_{k,k}}{a_{k-\htbs1\htbs,k}}
      \hat{z}_k\hs
    -\hs\frac{a_{k+1,k}}{a_{k-\htbs1\htbs,k}}
      \hat{z}_{k+1} \quad (n>k>1), \\
  \ds z_1\hts =\hs
      \hts1,\httbs\quad
  \ds z_2\hts =\hs
    -\hs\frac{a_{1,1}}{a_{2,1}}\hts,\httbs\quad
  \ds z_{k+1}\hts =\hs
    -\hs\frac{a_{k,k}}{a_{k+\htbs1\htbs,k}}
      z_k\hs
    -\hs\frac{a_{k-1,k}}{a_{k+\htbs1\htbs,k}}
      z_{k-1} \quad (1<k<n), \\
  \ds e_1\hts =\hs
      \hts1,\httbs\quad
  \ds e_{k+1}\hts =\hs
    \frac{a_{k+1,k}}{a_{k,k+1}}
      e_k \quad (1<k<n)\hts.
\end{array}\htts\right\}\httbs
\label{AlgLewis(1)}
\vspace{0.25ex}
\end{equation}
Then, the inverse $X = A^{-1}$ satisfies\vspace{-0.5ex}
\begin{equation}
\renewcommand{\arraystretch}{1.75}
\hspace*{-0.5ex}\left.\begin{array}{l}\vspace*{-5.5ex}\\
  \ds x_{1,n}\hs\htbs =\hs
    1\htts\big/\htts(a_{1,1}\hat{z}_1+a_{2,1}\hat{z}_2), \\
  \ds x_{s,k}\hs\hts =\hs
    (e_s z_s x_{1,n}) \hts \hat{z}_k \quad (s\leq k), \quad
  \ds x_{s,k}\hs =\hs
    (e_s \hat{z}_s x_{1,n})\hts z_k \quad (s>k), \quad (1\leq k\leq n).
\end{array}\htts\right\}\httbs
\vspace{-1ex}
\label{AlgLewis(2)}
\end{equation}
\end{theorem}
It is readily seen that the equations (\ref{AlgLewis(1)})--(\ref{AlgLewis(2)})
allow to compute the inverse matrix $X=A^{-1}$ using only $n^2+O(n)$ arithmetic operations. More
importantly, the algorithm based on Theorem \ref{TwLewis} uses every single scalar
equation (cf.\ (\ref{XAscalarRec}), (\ref{AXscalarRec}) and (\ref{NotMissingOnes}))
that results from the system of matrix equations $AX=I=X\httbs A$. Therefore, we
should expect the residual errors $\|AX-I\hts\|$ and $\|X\httbs A-I\hts\|$ to be
very small. Numerical experiments confirm that.

Unfortunately, the algorithm (\ref{AlgLewis(1)})--(\ref{AlgLewis(2)}) cannot be applied
for a quite wide class of tridiagonal matrices. The first problem is that the quantities
$z_k$ ($1\leq k\leq n$) and $\hat{z}_k$ ($n\geq k\geq 1$) may grow very fast causing
the floating-point overflow. In other words, the procedure will fail if $\fl(x_{1,n})=0$
with respect to the given floating-point arithmetic (it is quite easy to show, c.f.\ \cite{Lewis82},
that in theory, we always have $x_{1,n}\neq 0$ if the assumptions of Theorem \ref{TwLewis}
are satisfied). Note that all algorithms presented so far in this paper suffer
the high risk of floating-point overflow.

Another problem is the assumption that all sub- and super- diagonal elements
are different from 0. In \cite{Lewis82} the following solution is suggested. Assume
that $a_{k+1,k}=0$ for some $1\leq k < n$. Then, we have
\begin{equation}
A^{-1}\hs =\hs\hs \begin{pmatrix}
F & T\hs \\
0 & G\hs
\end{pmatrix}^{-1} =\hs\hs\hs
\begin{pmatrix}
F^{-1} & - F^{-1}T \hts G^{-1} \\
0 & G^{-1}
\end{pmatrix},
\label{BlockInv}
\end{equation}
where $F\in\mathbb{R}^{k\times k}$, $G\in\mathbb{R}^{(n-k)\times (n-k)}$
are tridiagonal matrices, and $T\in\mathbb{R}^{k\times (n-k)}$ has only one non-zero
element $t_{k,1}=a_{k,k+1}$ (if we assume that $a_{k,k+1}\neq 0$). Consequently,
\begin{equation*}
 - F^{-1}T \hts G^{-1} = - a_{k,k+1} [F^{-1}]_k \big([(G^{-1})^T]_1)^T ,  
\end{equation*}
where by $[\hs\cdot\hs]_k$ we denote
the $k$-th column of the corresponding matrix. The matrices $F$ and $G$ are inverted by
the algorithm (\ref{AlgLewis(1)})--(\ref{AlgLewis(2)}) (using the above block form again
whenever necessary). The numerical drawback of this approach is that the inverses of $G$ and $F$
are computed independently. As a result, the elements of the residual matrices $|AX-I|$ and
$|X\httbs A-I|$ that lay along the lines corresponding to the borders of the upper-right
block in (\ref{BlockInv}) may have the magnitude of order $\eps\max\{\cond{F},\cond{G}\}\htts\cond{A}$
(see the beginning of Section \ref{SecNewAlg} for more detailed explanation). Note that
there are matrices for which $\max\{\cond{F},\cond{G}\}$ is close to $\cond{A}$.

The equations (\ref{AlgLewis(2)}) delivers compact closed-form formulae for
the elements of the inverse matrix $X$. The following known characterisation of
the inverse of a tridiagonal matrix follows immediately.
\begin{theorem} \label{TwRank1}
If $A\in\mathbb{R}^{n\times n}$ is a tridiagonal matrix and
$X = A^{-1}$, then all matrices
\begin{equation*}
  U(k) = \{x_{s,j}\}_{\hs1\leq s\leq k,\hs k\leq j\leq n}\quad
\text{\textit{and}}\quad
  L(k) = \{x_{s,j}\}_{k\leq s\leq n,\hs 1\leq j\leq k}
\end{equation*}
($1\leq k\leq n$) have rank not greater than 1.
\end{theorem}
\begin{proof}
If $a_{k-1,k}\neq 0$ and $a_{k,k-1}\neq 0$ for $2\leq k\leq n$, the assertion is
obtained readily from (\ref{AlgLewis(2)}), otherwise from (\ref{BlockInv})
and (\ref{AlgLewis(2)}). Alternative proofs of the above property
can be found in \cite{Asplund59} and \cite{Faddeev81}.
\end{proof}

\section{The new methods} \label{SecNewAlg}
\setcounter{equation}{0}

We start this section by explaining why using all the scalar equations
that result from the condition $AX=I=X\httbs A$ is crucial for obtaining an algorithm
for inverting tridiagonal matrices that guarantees very small residual errors
$\|AX-I\hts\|$ and $\|X\httbs A-I\hts\|$. Assume that for some $1<s<n$ and $k>s$
the elements $x_{s,k-1}$, $x_{s,k}$, $x_{s,k+1}$ of the inverse matrix $X$
were evaluated as $x_{s,k+j} = c_s v_{k+j}$ ($j=-1, 0, 1$)
for some real number $c_s$, and that the quantity $v_{k-1}$
was computed from the equation
\begin{equation*}
  v_{k-1}\hts =\hs
    -\hs\frac{a_{k,k}}{a_{k-\htbs1\htbs,k}}
      v_k\hs
    -\hs\frac{a_{k+1,k}}{a_{k-\htbs1\htbs,k}}
      v_{k+1} .
\vspace{-1ex}
\end{equation*}
In such a case, we have\vspace{-0.5ex}
\begin{equation*}
  a_{k-1,k}\hts x_{s,k-1} = - a_{k,k}\hts x_{s,k}(1+\gamma_0) - a_{k+1,k}\hts x_{s,k+1}(1+\gamma_1) ,
\end{equation*}
where $|\gamma_0|,|\gamma_1| \leq 4\eps$ (for simplicity we ignore
the terms of order $\eps^2$). Consequently,
\begin{equation}
  \big| a_{k-1,k}\hts x_{s,k-1} + a_{k,k}\hts x_{s,k} + a_{k+1,k}\hts x_{s,k+1} \big| \hs\leq\hs
  4\hts\eps\big(|a_{k,k}\hts x_{s,k}|+|a_{k+1,k}\hts x_{s,k+1}|\big) .
\label{ScalarResErrGoodRec}
\end{equation}
On the other hand, if at least one of the elements $x_{s,k-1}$, $x_{s,k}$, and $x_{s,k+1}$
was computed independently of the other two, then the best we may in general expect
is that $x_{s,k+j} = y_{s,k+j}(1+\beta_j\htts\cond{A})$ ($j=-1,0,1$), where $y_{s,k+j}$
denote the exact values of the considered elements of $A^{-1}$, and $|\beta_j|\leq C\eps$
for some small $C>0$. We have, of course,
$a_{k-1,k}\hts y_{s,k-1} + a_{k,k}\hts y_{s,k} + a_{k+1,k}\hts y_{s,k+1}=0$, however,
if $y_{s,k+j}$ ($j=-1,0,1$) are replaced by their computed values $x_{s,k+j}$,
then we obtain\vspace{-1ex}
\begin{equation}
  \big| a_{k-1,k}\hts x_{s,k-1} + a_{k,k}\hts x_{s,k} + a_{k+1,k}\hts x_{s,k+1} \big| \hs\leq\hs
  C\eps\htts\cond{A}\sum_{j=-1}^{1}|a_{k+j,k}\hts y_{s,k+j}|\hts.
\label{ScalarResErrBadRec}
\vspace{-0.5ex}
\end{equation}
By comparing the inequalities (\ref{ScalarResErrGoodRec}) and (\ref{ScalarResErrBadRec}),
we conclude that the algorithm that does not \emph{fully} exploit the equations
$AX=I=X\httbs A$ may compute the inverse for which the residual error $\|X\httbs A-I\|$
(or $\|AX-I\|$) is about $\cond{A}$ times larger than in the case of the algorithm that does.
Finding an algorithm which is a complete reflection of the equations $AX=I=X\httbs A$, and
can be applied for an arbitrary tridiagonal matrix $A$ (and is stable) is not easy,
and --- to our knowledge --- has not been succeeded yet.

Before we take care of the problem described above, we will solve a little less difficult one,
related to the possible occurrence of the floating point overflow when computing the
solutions of the recurrence relations.

\subsection{The simple ratio-based method} \label{SubSecRatio}

Assume that the solution of the difference equation (\ref{GenDiffEq}) we are looking for
satisfies the condition $x_k\neq 0$ for all $k\geq 0$. If we define $\ds r_k = x_{k}^{-1}x_{k-1}$,
then the equation (\ref{GenDiffEq}) can be written in the equivalent form:
\begin{equation*}
  a_{-\htbs1}(k)\htts r_k\hs +\hs
    a_{1}(k)\hts\frac{1}{r_{k+1}}\hs =\hs -\hts a_{0}(k)
    \qquad ( k > 1,\hs\hs a_{\pm1}(k)\neq 0)\hts.
\end{equation*}
Once the ratios $r_k$ are computed, and the value $x_1$ or $x_n$ (for some $n>1$)
is known, we may easily compute all other values $x_k$:
\begin{equation*}
  x_k\hts =\hs \frac{x_{k-1}}{r_k}
  \quad\hs\hs \text{or}\quad\hs\hs
  x_{k-1}\hts =\hs r_k\hts\hts x_k \qquad (1 < k \leq n)\hts.
\end{equation*}
This approach is well know in the theory of the second order difference
equations and was described in detail in \cite{Gautschi67}. It practically
eliminates the risk of the floating point overflow when computing the
solution $\{x_k\}$ of (\ref{GenDiffEq}). The use of ratios in the
problem of inverting tridiagonal matrices was already suggested (but not strictly
formulated\footnote{The strict formulation was given later in \cite{Jain07},
but the formulae for the diagonal elements of $X$ given there
are numerically unstable in general case%
.}) in \cite{Golub85}, where the algorithm, in a sense similar to
(\ref{AlgLewis(1)})--(\ref{AlgLewis(2)}), for inverting tridiagonal
symmetric positive definite matrices with all negative sub- and super-
diagonal elements was proposed.

Let us consider the lower triangle of $X = A^{-1}$, and assume that
$x_{ij}\neq 0$\, ($1\leq i\leq n$, $1\leq j\leq n$. From the equation $X\httbs A = I$,
we have (for $1 < s\leq n$ and $1 < k < s$)\vspace{-0.5ex}
\begin{align*}
  & a_{1,1}\hs x_{s,1} + a_{2,1}\hs x_{s,2} = 0\hs, \\[0.5ex]
  & a_{k-1,k}\hs x_{s,k-1} + a_{k,k}\hs x_{s,k} +
    a_{k+1,k}\hs x_{s,k+1} = 0\hts .
\vspace{-0.5ex}
\end{align*}
Now, setting $q_k = x_{s,k-1}\big/x_{s,k}$ ($1 < k\leq n$), we immediately obtain\vspace{-0.5ex}
\begin{equation}
  q_2 = - a_{2,1} \big/ a_{1,1}, \quad
  q_{k+1} = -a_{k+1,k} \big(a_{k,k} + a_{k-1,k}\hts q_{k}\big)^{-1}
    \quad (1 < k < n)\hts.
\label{ColRatiosRec}
\end{equation}
From the equation $AX = I$, we have\vspace{-0.5ex}
\begin{equation}
  r_{n-1} = -a_{n,n-1} \big/ a_{n,n}, \quad
  r_{k-1} = -a_{k,k-1} \big(a_{k,k} + a_{k,k+1}\hts r_{k}\big)^{-1}
    \quad (n > k > 1)\hts,
\label{RowRatiosRec}
\end{equation}
where $r_k = x_{k+1,j}\big/x_{k,j}$ ($1\leq k < n$, $1\leq j\leq k$). Note that
the ratios $r_k$ do not depend on the column index $j$, and $q_k$ ($2\leq k\leq n$) do not
depend on the row index $s$ (this fact may be considered as another proof of Theorem \ref{TwRank1}).

For the upper triangle of $X$, we define
$\hat{r}_k = x_{k-1,j}\big/x_{k,j}$ ($1 < k\leq n$, $k\leq j\leq n$) and
$\hat{q}_k = x_{s,k+1}\big/x_{s,k}$ ($1\leq k < n$, $1\leq s\leq k$). For these ratios,
the recurrences similar to (\ref{ColRatiosRec}) and (\ref{RowRatiosRec}) can be also derived.
However, the upper triangle ratios should not be computed independently of the lower
triangle ones for the reasons described earlier --- the residual errors may depend on
$\cond{A}^2$ in such a case. The following result should be applied instead.
\begin{lemma} \label{LmRatiosSymm}
Let $A$ be a tridiagonal matrix satisfying $a_{k-1,k}\neq 0$, $a_{k,k-1}\neq 0$
for $2\leq k\leq n$. If the ratios $r_k$, $\hat{q}_k$ ($1\leq k\leq n-1$), and $q_k$, $\hat{r}_k$
($2\leq k\leq n$) are defined as above, then
\begin{equation}
  a_{k+1,k}\hts\hat{q}_k \hs=\hs a_{k,k+1}\hts r_k , \qquad
  a_{k,k-1}\hts\hat{r}_k \hs=\hs a_{k-1,k}\hts q_k .
\label{RatiosSymm}
\end{equation}
\end{lemma}
\begin{proof}
The proof follows readily from (\ref{LewisSymm}).
\end{proof}

Observe that from the equation
\begin{equation}
  a_{n-1,n}\hs x_{n,n-1} + a_{n,n}\hs x_{n,n}\hs =\hs 1
\label{StartEq}
\end{equation}
we immediately obtain that $x_{n,n} = (a_{n-1,n}q_n + a_{n,n})^{-1}$. If a tridiagonal
matrix $A$ satisfies the assumptions of Lemma \ref{LmRatiosSymm}, and its inverse has only
non-zero elements, then from (\ref{StartEq}) and (\ref{ColRatiosRec})--(\ref{RatiosSymm}), we
obtain the following set of algorithms for inverting the matrix $A$.
\begin{algorithm}[the set of algorithms] \label{AlgKW} ~\\[-3ex]
\begin{enumerate}{\rm
\renewcommand{\itemsep}{0ex}
\item Compute the ratios $q_k$ and $r_k$ according to (\ref{ColRatiosRec})--(\ref{RowRatiosRec}).
\item Compute the ratios $\hat{q}_k$ and $\hat{r}_k$ by applying (\ref{RatiosSymm}).
\item Set $x_{n,n} = (a_{n-1,n}q_n + a_{n,n})^{-1}$ (or use the analogous formula for $x_{1,1}$).
\item Using the ratios $r_k$, $q_k$, $\hat{r}_k$, and $\hat{q}_k$, compute all other elements
of $X$ from the adjacent ones, in an (theoretically) arbitrary order, using
only one multiplication or division for each element.}
\end{enumerate}
\end{algorithm}

Clearly, the above algorithm uses $n^2+O(n)$ arithmetic operations.
In Figure~\ref{FigOrder}, we present some exemplary orders the inverse matrix $X$ may
be computed in. The first of them is presented as an analogy to the algorithm based on (\ref{SimplRecAlg}).
This time, however, no instability occurs, as the recurrences for ratios are carried out in
the stable direction (the arrows on the graphs correspond to Step 4 of Algorithm
\ref{AlgKW}). The last example is presented just for fun. Note that with these two
orders, the algorithm fails if $\fl(x_{1,n})=0$. There is no such problem with the two
remaining suggested orders. The lower left one is much better suited in the case
matrices are stored column-wise in the computer memory (one may use,
of course, its row analogy if matrices are stored row-wise).
\begin{figure}[h!]
\begin{center}\vspace{1ex}
\includegraphics[width=5cm,height=5cm,angle=0]{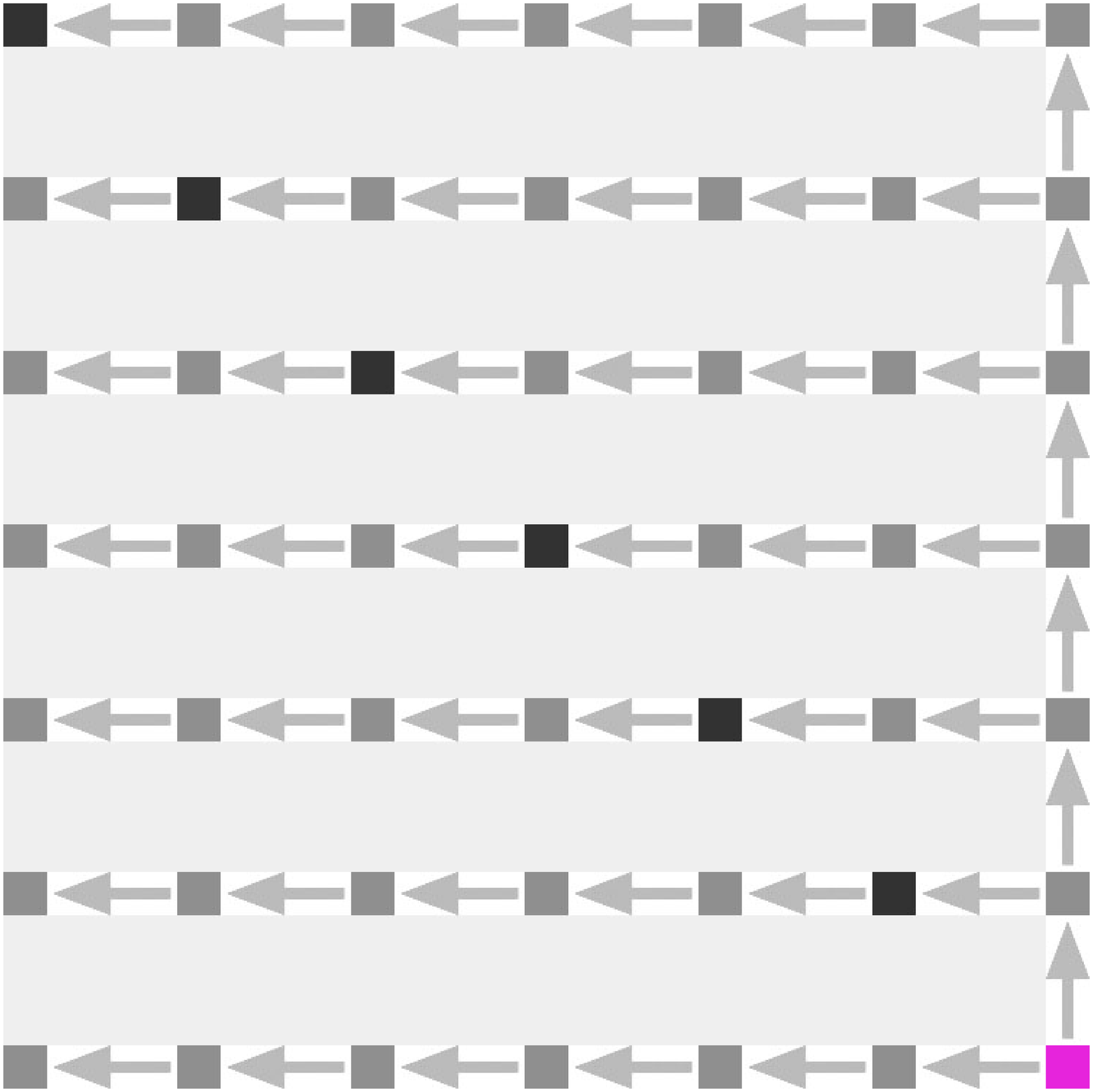}\hspace*{4ex}
\includegraphics[width=5cm,height=5cm,angle=0]{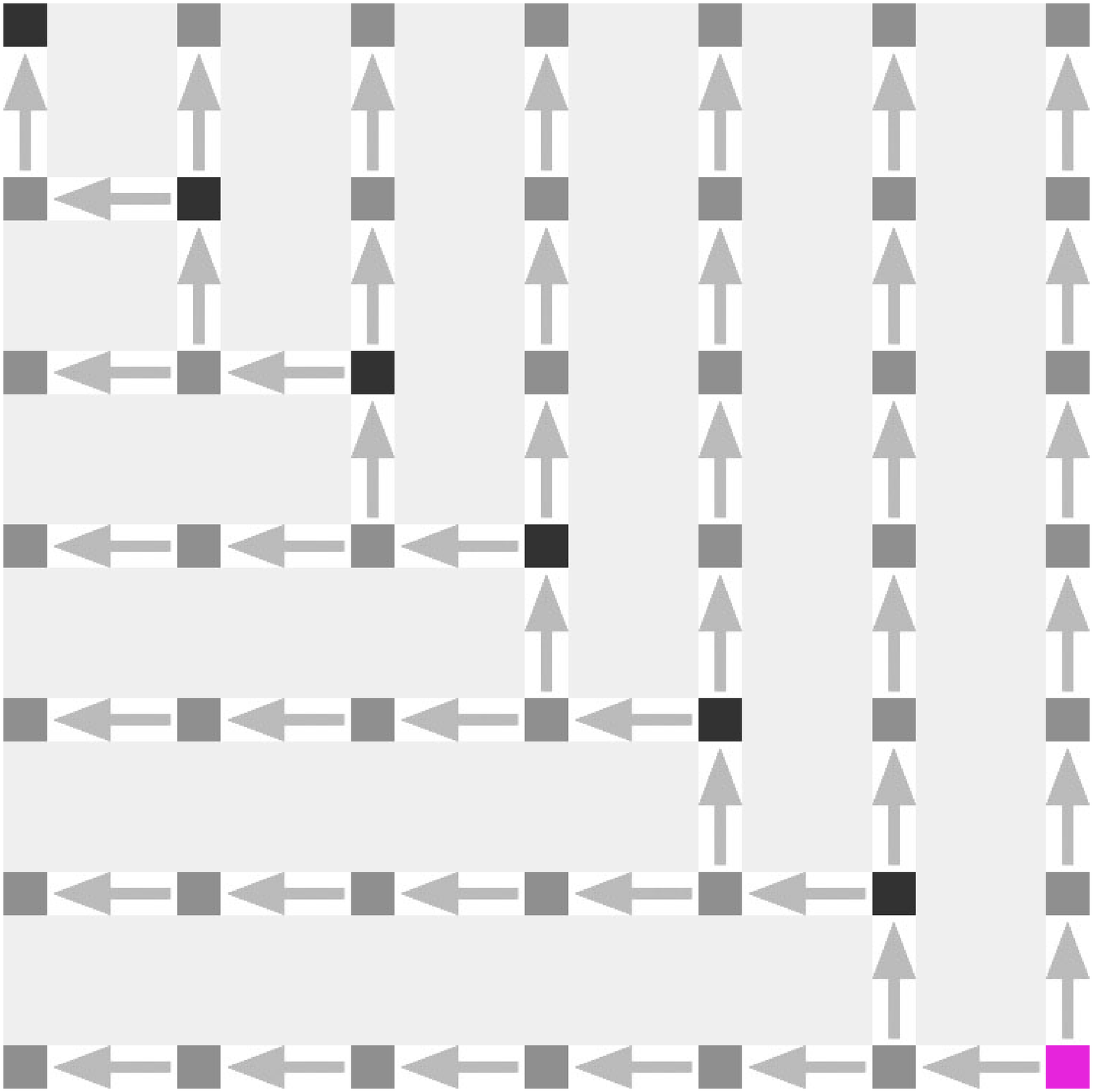}\\[4ex]
\includegraphics[width=5cm,height=5cm,angle=0]{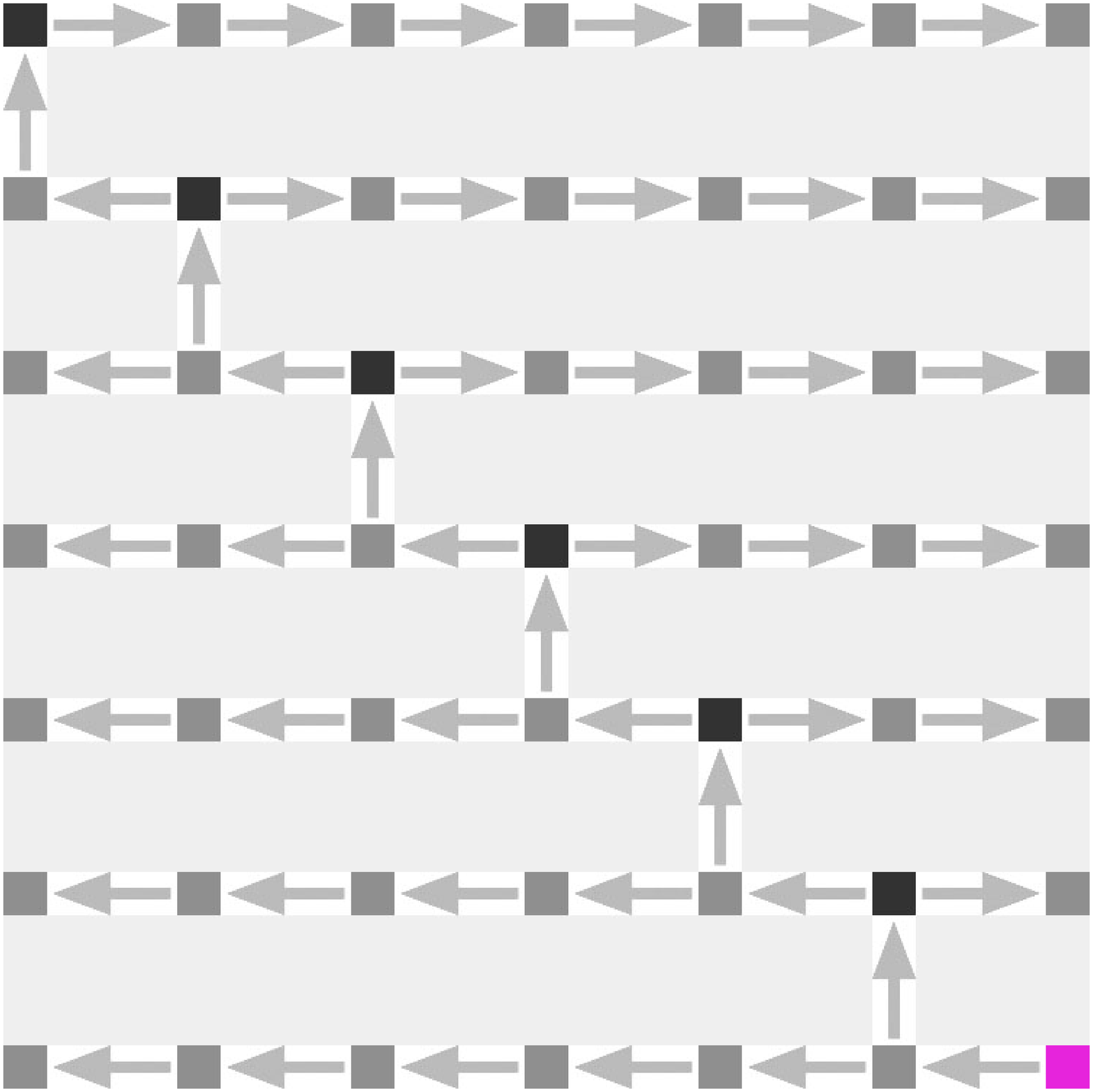}\hspace*{4ex}
\includegraphics[width=5cm,height=5cm,angle=0]{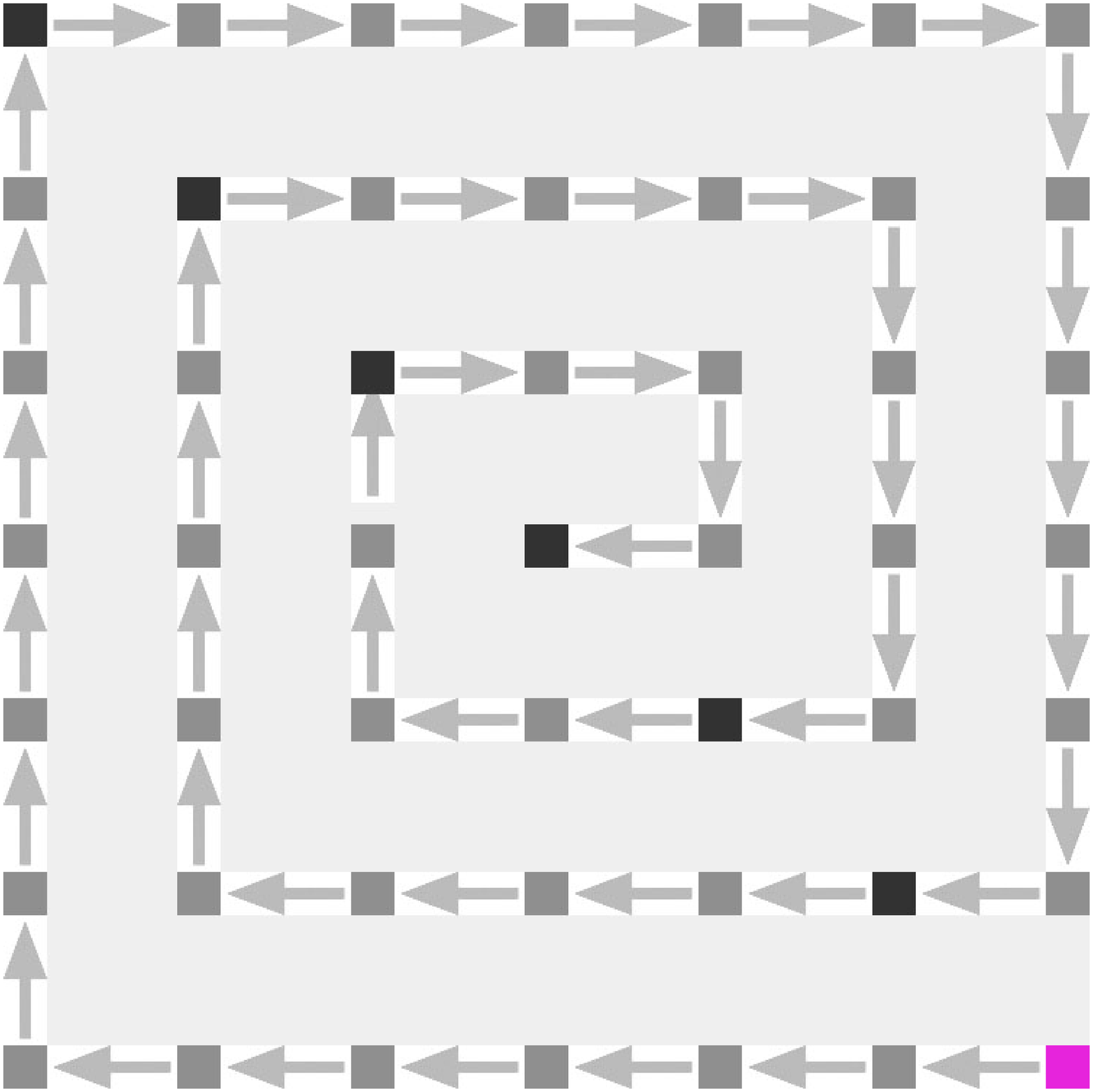}\vspace{-2ex}
\end{center}
\caption{\small Four examples of many different orders the elements of the inverse $X$
may be computed when using Algorithm \ref{AlgKW}. The initial element is coloured purple.}
\label{FigOrder}
\end{figure}

In the remaining part of the paper we shall use and extend the following
formulae, which corresponds to the lower left diagram
in Figure \ref{FigOrder}:\vspace{-0.25ex}
\begin{equation}
\renewcommand{\arraystretch}{2.25}
\left.\begin{array}{l}\vspace*{-7.5ex}\\
\ds
  q_2 = \frac{-a_{2,1}}{a_{1,1}}, \quad
  q_{k+1} = \frac{-a_{k+1,k}}{a_{k,k}+a_{k-1,k} q_k} \quad (1 < k < n)\hts , \\
\ds
  r_{n-1} = \frac{-a_{n,n-1}}{a_{n,n}}, \quad
  r_{k-1} = \frac{-a_{k,k-1}}{a_{k,k}+a_{k,k+1} r_k} \quad (n > k > 1)\hts , \\
\ds
  x_{n,n} = \big(a_{n-1,n}q_n + a_{n,n}\big)^{-1}\hts , \\[1.75ex]
\ds
  \hspace*{-1.225ex}
  {\renewcommand{\arraystretch}{1.5}
   \left.\begin{array}{l}
     x_{j,k-1} = q_k\htts x_{j,k} \quad (k\leq j\leq n) \\
     x_{k-1,k-1} = r_{k-1}^{-1} x_{k,k-1}
   \end{array}\right\}} \quad (n \geq k > 1)\hts , \\[2.5ex]
\ds
  x_{j,k+1} = \big(a_{k+1,k}^{-1} a_{k,k+1} r_{k}\big) x_{j,k}
   \quad  (1\leq j\leq k,\,\,\,1 \leq k < n)\hts .\\[-1ex]
\end{array}\quad\right\}
\label{KWbasic}
\vspace{0.75ex}
\end{equation}
The above scheme is valid only if the inverse matrix $X$ has only non-zero elements in its
lower triangle (note that this implies that $a_{k+1,k}\neq 0$ for $1\leq k < n$,
and that each ratio that appears in (\ref{KWbasic}) is finite --- for proof,
see Theorem \ref{TwZeroBlocks} below). Now, we shall investigate the
numerical properties of the algorithm based on (\ref{KWbasic}).

Consider the elements $x_{s,k-1}$, $x_{s,k}$, $x_{s,k+1}$ ($k<s$) of $X$. They are computed
as follows $x_{s,k} = q_{k+1} x_{s,k+1}$, $x_{s,k-1} = q_{k}\hts x_{s,k}$, which means that
the numerically computed values satisfy $x_{s,k+1}^{-1}x_{s,k} = q_{k+1}(1+\gamma_{k+1})$ and
$x_{s,k}^{-1}\hts x_{s,k-1} = q_{k}(1+\gamma_{k})$, where $|\gamma_{k}|,|\gamma_{k+1}|\leq\eps$.
From the first set of equations in (\ref{KWbasic}), we readily obtain that for the numerically
computed ratios $q_k$ and $q_{k+1}$ the following equality holds:
\begin{equation*}
  -a_{k+1,k}\hts q_{k+1} = a_{k,k}(1+\alpha_k) + a_{k-1,k}\hts q_{k}(1+\beta_k)
  \qquad (\hts|\alpha_k|\leq 2\eps,\,\,\, |\beta_k|\leq 3\eps)\hts
\end{equation*}
(again, we ignore the terms of order $\eps^2$).
By combining the two above results, we get
\begin{equation}
 \big| a_{k+1,k}\hts x_{s,k+1} + a_{k,k}\hts x_{s,k} + a_{k-1,k}\hts x_{s,k-1}\big|\hs \leq\hs
   4\eps\big(|a_{k,k}\hts x_{s,k}| + |a_{k-1,k}\hts x_{s,k-1}|\big) .
\label{KWlowRowEqEst}
\end{equation}
Every three adjacent elements in a column of the lower triangle of $X$ are computed
exactly as follows:\vspace{-1ex}
\begin{align*}
  x_{k-1,j}\hs &=\hs x_{k,k}\hts q_{k}/r_{k-1}\hts q_{k-1}\hts q_{k-2} \cdots\hts q_{j+1} \\[0.5ex]
  x_{k,j}\hs &=\hs x_{k+1,k}/r_k\hts q_{k}\hts q_{k-1}\htbs\cdots\hts q_{j+1} , \\[0.5ex]
  x_{k+1,j}\hs &=\hs x_{k+1,k}\hts q_{k}\hts q_{k-1}\htbs\cdots\hts q_{j+1} .
\vspace{0.5ex}
\end{align*}
Therefore, a bound analogous to (\ref{KWlowRowEqEst}) can be obtained with the leading
factor $4$ replaced with some $c(j)$ that depends linearly on the column index $j$. Obviously,
$c(j) \leq C(n)$ for some $C(n)=O(n)$. Similar $O(n)$-dependent bounds may be obtained
(in a little more tedious way) for all remaining elements of the matrices $|X\htbs A-I|$
and $|AX-I|$ if we recall the relations (\ref{RatiosSymm}).

In order to continue estimating the residual errors, we need the following,
quite easy to prove lemma.

\begin{lemma} \label{LmSumNorm}
Let two matrices $F\in\mathbb{R}^{n\times n}$, $G\in\mathbb{R}^{n\times n}$ satisfy
\begin{equation*}
  |f_{ij}|\hts \leq\hs |g_{ij}| + |g_{i,j+1}| \qquad (1\leq i\leq n,\,\, 1\leq j\leq n)\hts ,
\end{equation*}
where we assume that $g_{i,n+1}=0$ ($1\leq i\leq n$). Then,
\begin{equation}
  \| F\htts \|_1\hs \leq\hs\hts 2\hts \|G\hts \|_1\hts , \quad
  \| F\htts \|_{\infty}\hs \leq\hs\hts 2\hts \|G\hts \|_{\infty}\hts , \quad
  \| F\htts \|_{F}\hs \leq\hs\hts 2\hts \|G\hts \|_{F}\hts .
\label{SumNormEst}
\end{equation}
\end{lemma}

From the above lemma (or its several obvious generalisations),
we immediately obtain that the inverse matrix $X$ computed numerically
using the algorithm (\ref{KWbasic}) satisfies (for simplicity, we restrict
our attention to the $\|\cdot\|_1$ norm only)
\begin{equation*}
  \max\big\{\| X\htbs A - I\hts \|_1,\, \| AX - I\hts \|_1\big\}\hs \leq\hs \eps K(n)\|A\|_1 \|X\|_1 ,
\end{equation*}
where $K(n) = O(n)$. Note that a similar (even sharper) bound for $\|X\htbs A -I\|_1$ holds
in the case of the unstable recursive algorithm based on (\ref{SimplRecAlg}). The important difference
is the relation between $\|X\|_1$ and $\|A^{-1}\|_1$. We are convinced that if the inverse $X$
is computed using (\ref{KWbasic}), then
\begin{equation}
  \| X \|_1\hs \leq \hs \| A^{-1} \|_1 \hts \big( 1 + \eps J(n)\hts \cond{A} \big)
\label{Xnorm}
\end{equation}
for some $J(n) = O(n)$. Possibly, the difference between $\|A^{-1}\|_1$ and $\|X\|_1$ may be larger
than the right hand side of (\ref{Xnorm}), but then $\|X\|_1\htbs<\|A^{-1}\|_1$. Our presumption
is based on the fact that we are computing minimal solutions of the difference equations which
correspond to the equations $AX=I=X\htbs A$ (the recurrences for ratios are carried out in
the stable, towards-the-diagonal direction). In such a case, if we are moving away (along
a row or a column) from the main diagonal, the elements of $X$ should not grow
faster (or decrease slower) in modulus than the elements of $A^{-1}$, which,
we think, implies the inequality (\ref{Xnorm}). However, a formal proof
of that assumption may be difficult.

Assuming that (\ref{Xnorm}) holds, we obtain the following:

\begin{conjecture} \label{CjKWerr}
If $A\in\mathbb{R}^{n\times n}$ is a non-singular matrix whose inverse has only non-zero elements
in the lower triangle, and $\htts\eps\hts n\htts \cond{A}\httbs<\htbs C$ for some constant $C$,
then the inverse $X$ computed numerically by the algorithm (\ref{KWbasic}) satisfies
\begin{equation}
  \max\big\{ \| X\htbs A - I\hts \|_1,\htts \| AX - I\hts \|_1\big\}\htts \leq\hs \eps K(n)\hts \cond{A} ,  
\label{MainErrEst}
\end{equation}
for some $K(n)=O(n)$.
\end{conjecture}

\begin{remark} \label{RmLewisErr}
{\rm In the case of the algorithm of Lewis (c.f.\ \cite{Lewis82} or Theorem \ref{TwLewis}),
bounds similar to (\ref{KWlowRowEqEst}) can be derived for all non-diagonal elements of the residual matrices
$AX-I$ and $X\htbs A-I$. For the diagonal elements, only $O(n)$-dependent bounds hold. This
implies that the numerical properties of the algorithm (\ref{AlgLewis(1)})--(\ref{AlgLewis(2)})
and the new algorithm based on (\ref{KWbasic}) are comparable. Recall that if the Lewis
algorithm is used together with (\ref{BlockInv}), then moduli of some elements of
the residual matrices may be almost as large  as $\eps\hts\cond{A}^2$.}
\end{remark}

\begin{remark} \label{Rm2ndNorm}
{\rm The important step in justifying Conjecture \ref{CjKWerr} is Lemma \ref{LmSumNorm} which is not
true in the case of the second matrix norm in general. The assertion of Lemma \ref{LmSumNorm} is
an immediate consequence of the fact that $\|G\|=\||G|\|_{\triangle}$ if $\triangle\in\{1,\infty,F\}$.
In the case of the $\|\cdot\|_2$ norm, taking the moduli of all elements of a matrix may
increase the norm by a factor proportional to $\sqrt{n}$ ($n$ is the matrix size). However,
if $G=A$, where $A$ is a tridiagonal matrix, or $G=A^{-1}$, then $\|G\|\leq 2\||G|\|_2$, and
the inequality similar to (\ref{SumNormEst}) is satisfied with the factor $2$ replaced by $4$.
Consequently, we suspect the error estimation (\ref{MainErrEst}) to be also true in the case
of the second matrix norm.}
\end{remark}

\subsection{The extended ratio-based method}

The last thing to do is to extend the scheme given in (\ref{KWbasic}) so that it
can be applied to an arbitrary non-singular tridiagonal matrix. Obviously, unlike, e.g.,
(\ref{BlockInv}), the extension should preserve the very favourable numerical properties of
the formulae (\ref{KWbasic}), and also should not increase the computational complexity. To achieve
the goal, several conditions need to be fulfilled. The computations may include only one initial element,
all other elements of the inverse matrix $X$ should be computed from another element (adjacent if possible)
by only one multiplication (or division), $x_{ij} = c_{ij}x_{i-q,j-r}$, in such a way that each relation
between two elements is a direct consequence of the equations $AX\htbs=\htbs I\htbs=\htbs X\htbs A$.
Of course, the number of additional arithmetic operations should depend linearly on the
matrix size $n$.

In order to do so, we need a complete characterisation of possible
shapes of blocks of zeros in the inverse matrix $X$.

\begin{theorem} \label{TwZeroBlocks}
Let $A\in\mathbb{R}^{n\times n}$ be a non-singular tridiagonal matrix, and $X=A^{-1}$.\\[0.5ex]
{\rm\textbf{a)}} If $\hts a_{k+1,k}=0$,
then $x_{ij} = 0$ for all $\hts k+1 \leq i \leq n$ and $1\leq j\leq k$ ($1\leq k \leq n-1$).\\[0.5ex]
{\rm\textbf{b)}} If $\hts a_{k-1,k}=0$,
then $x_{ij} = 0$ for all $\hts 1 \leq i \leq k-1$ and $k\leq j\leq n$ ($2\leq k \leq n$).\\[0.5ex]
{\rm\textbf{c)}} If $\hts |q_{k}|=\infty$ or, equivalently, $|\hat{r}_{k}|=\infty$, then
$\hts x_{s,k} = 0$ for $\hts k \leq s \leq n$, and $\hts x_{k,j} = 0$ for $\hts k \leq j \leq n$
($2\leq k \leq n$).\\[0.5ex]
{\rm\textbf{d)}} If $\hts |r_{k}|=\infty$ or, equivalently, $|\hat{q}_{k}|=\infty$, then
$\hts x_{k,j} = 0$ for $\hts 1 \leq j \leq k$, and $\hts x_{s,k} = 0$ for $\hts 1 \leq s \leq k$
($1\leq k \leq n-1$).\\[0.5ex]
{\rm\textbf{e)}} Each block of zeros in the inverse matrix $X$ may consist only
of the four different types of blocks described above.
\end{theorem}
\begin{proof}
The above assertions are a direct consequence of Theorem \ref{TwRank1} and (\ref{BlockInv}).
\end{proof}

Before we proceed, one more problem, related to the computation of ratios, has to be solved.
Observe that if $a_{k,k-1}=0$, then $q_{k}=0$. In this case, the relation (\ref{RatiosSymm})
remains true, but it cannot be used to compute the ratio
$\hat{r}_k = a_{k,k-1}^{-1}a_{k-1,k}\hts q_{k}$.
The following lemma delivers a scheme that allows to compute all the
ratios in every case, without sacrificing the relation (\ref{RatiosSymm}).

\begin{lemma} \label{LmRatiosSmart}
{\bf a)} The ratios $q_k = x_{j,k}^{-1}\hts x_{j,k-1}$ and $\hat{r}_k = x_{k,j}^{-1}\hts x_{k-1,j}$
($1<k\leq n$, $j\geq k$) satisfy\vspace{-0.25ex}
\begin{equation}
\renewcommand{\arraystretch}{1.4}
  \hspace*{-1.5ex}\left.
  \begin{array}{l}\vspace{-4.5ex}\\
    q_{2} = -a_{1,1}^{-1}\hts a_{2,1}\hts, \qquad q_{k+1} = -s_k^{-1}\hts a_{k+1,k}\hts,\\
    \hat{r}_{2} = -a_{1,1}^{-1}\hts a_{1,2}\hts, \qquad \hat{r}_{k+1} = -s_k^{-1}\hts a_{k,k+1}\hts,
  \end{array}\quad\right\} \quad (1 < k < n)
\label{RatiosSmartQLRU}
\vspace{-1ex}
\end{equation}
where\vspace{-1ex}
\begin{equation}
  s_{k} = a_{k,k}+a_{k-1,k}\hts q_{k} \qquad (1 < k < n)\hts.
\label{sQLRU}
\vspace{1.5ex}
\end{equation}
{\bf b)} The ratios $\hat{q}_k = x_{j,k}^{-1}\hts x_{j,k+1}$ and $r_k = x_{k,j}^{-1}\hts x_{k+1,j}$
($1\leq k< n,\,\,j\leq k$) satisfy\vspace{0.125ex}
\begin{equation}
\renewcommand{\arraystretch}{1.4}
  \hspace*{-1.5ex}\left.
  \begin{array}{l}\vspace{-4.5ex}\\
    \hat{q}_{n-1} = -a_{n,n}^{-1}\hts a_{n-1,n}\hts, \qquad \hat{q}_{k-1} = -t_k^{-1}\hts a_{k-1,k}\hts,\\
    r_{n-1} = -a_{n,n}^{-1}\hts a_{n,n-1}\hts, \qquad r_{k-1} = -t_k^{-1}\hts a_{k,k-1}\hts,
  \end{array}\quad\right\} \quad (n > k > 1)
\label{RatiosSmartQURL}
\vspace{-1ex}
\end{equation}
where\vspace{-1ex}
\begin{equation}
  t_{k} = a_{k,k}+a_{k,k+1}\hts r_k \qquad (2 \leq k < n)\hts.
\label{tQURL}
\vspace{1ex}
\end{equation}
\end{lemma}
\begin{proof}
The formulae (\ref{RatiosSmartQLRU}) and (\ref{RatiosSmartQURL}) follow immediately from the
definition of the ratios and the relation (\ref{RatiosSymm}) if we additionally assume the standard
convention that $c/0 = \infty$ for $c>0$, $c/0 = -\infty$ for $c<0$, and $c/\infty = 0$ for $|c|<\infty$.
\end{proof}

\begin{remark}
{\rm Note that the values $s_k$ and $t_k$ do not have to be remembered, i.e.\ in practical
implementation, can be replaced by a single variable.}
\end{remark}

\begin{remark}
{\rm By (\ref{RatiosSymm}), the last terms in (\ref{sQLRU}) and (\ref{tQURL}) can be replaced with
$a_{k,k-1}\hts \hat{r}_{k}$ and $a_{k+1,k}\hts \hat{q}_k$, respectively. The method proposed in
this paper has slightly better numerical properties if the ''leading'' ratios --- the ones that appear
in (\ref{sQLRU}) and (\ref{tQURL}) --- belong to the same triangle, i.e. if the leading pairs
are: $\{q_k\}$ and $\{r_k\}$, or\hts: $\{\hat{q}_k\}$ and $\{\hat{r}_k\}$.}
\end{remark}

Now, we may formulate the main result of this Section.

\begin{theorem} \label{TwKWext}
Let $A\in\mathbb{R}^{n\times n}$ be a non-singular tridiagonal matrix, and $X=A^{-1}$.
Let us assume that the ratios $q_{k}$, $\hat{r}_k$ ($2\leq k\leq n$), and $\hat{q}_k$, $r_{k}$
($1\leq k\leq n-1$) are given by (\ref{RatiosSmartQLRU}) and (\ref{RatiosSmartQURL}).\\[0.5ex]
{\rm\textbf{a)}} For $2\leq k\leq n$, if $|q_{k}| < \infty$, then\vspace{-0.5ex}
\begin{equation}
  x_{s,k-1} = q_{k}\hts x_{s,k}  \quad (k \leq s\leq n)\hts ,
\label{KWextLower}
\vspace{-0.5ex}
\end{equation}
otherwise, for \hts $k < s\leq n$, we have\vspace{-0.5ex}
\begin{equation}
  x_{s,k-1}\htts = \left\{
  \renewcommand{\arraystretch}{1.4}
  \begin{array}{ll}\vspace*{-4.125ex}\\
    \htbs
    a_{n-1,n}^{-1}
      &\,\mathrm{if}\quad\httbs k = n\hts, \\
    \htbs
    -a_{k+1,k} a_{k-1,k}^{-1}\hts x_{s,k+1}
      &\,\mathrm{if}\quad\httbs k < n\hts .
  \end{array}\right.
\label{KWextLowerQinf}
\end{equation}
In addition, if $|q_{k}| = \infty$, then for $2\leq k < n$,
\begin{equation}
  x_{k,k-1}\htts = \left\{
  \renewcommand{\arraystretch}{1.4}
  \begin{array}{ll}\vspace*{-4.33ex}\\
    \htbs
    r_{k}^{-1} x_{k+1,k-1}
      &\,\mathrm{if}\quad\httbs r_{k}\neq 0 \hts, \\
    \htbs
    -a_{k+1,k+2} a_{k+1,k}^{-1}\hts x_{k+2,k-1}
      &\,\mathrm{if}\quad\httbs r_{k} = 0,\, a_{k+1,k}\neq 0 \hts, \\
    \htbs
    -a_{k,k+1} a_{k-1,k}^{-1}\hts \hat{q}_{k}^{\htts-1} x_{k+1,k+1}
      &\,\mathrm{if}\quad\httbs r_{k} = 0,\, a_{k+1,k} = 0,\, \hat{q}_{k}\neq 0 \hts, \\
    \htbs
    a_{k+1,k+2} a_{k-1,k}^{-1}\hts x_{k+2,k+1}
      &\,\mathrm{if}\quad\httbs r_{k} = 0,\, a_{k+1,k} = 0,\, \hat{q}_{k} = 0,\, a_{k,k+1}\neq 0 \hts, \\
    \htbs
    a_{k-1,k}^{-1}
      &\,\,\mathrm{otherwise}\hts .
  \end{array}\right.
\label{KWextLowerMissing}
\end{equation}
{\rm\textbf{b)}} For $2\leq k\leq n$, the diagonal element $\hts x_{k-1,k-1}$ satisfies
\begin{equation}
  x_{k-1,k-1}\htts = \left\{
  \renewcommand{\arraystretch}{1.4}
  \begin{array}{ll}\vspace*{-4.33ex}\\
    \htbs
    r_{k-1}^{-1} x_{k,k-1}
      &\,\mathrm{if}\quad\httbs r_{k-1}\neq 0\hts , \\
    \htbs
    \hat{r}_k\hts\hat{q}_{k-1}^{\htts-1}x_{k,k}
      &\,\mathrm{if}\quad\httbs r_{k-1} = 0,\, \hat{q}_{k-1}\neq 0\hts , \\
    \htbs
    -a_{k,k+1} a_{k,k-1}^{-1} x_{k+1,k-1}
      &\,\mathrm{if}\quad\httbs r_{k-1} = 0,\, \hat{q}_{k-1} = 0,\, a_{k,k-1}\neq 0\hts , \\
    \htbs
    -a_{k,k+1} a_{k-1,k}^{-1}\hts \hat{r}_k\hts x_{k+1,k}
      &\,\mathrm{if}\quad\httbs r_{k-1} = 0,\, \hat{q}_{k-1} = 0,\, a_{k,k-1} = 0,\, a_{k-1,k}\neq 0\hts , \\
    \htbs
    \big(a_{k-2,k-1} q_{k-1} + a_{k-1,k-1}\big)^{-1}
      &\,\mathrm{otherwise}\hts .
  \end{array}\right.
\label{KWextDiag}
\end{equation}
{\rm\textbf{c)}} For $1\leq k\leq n-1$, if $|\hat{q}_{k}| < \infty$, then\vspace{-0.5ex}
\begin{equation}
  x_{s,k+1} = \hat{q}_{k}\hts x_{s,k}  \quad (1 \leq s\leq k)\hts ,
\label{KWextUpper}
\vspace{-0.5ex}
\end{equation}
otherwise, for \hts $1 \leq s < k$ we have\vspace{-0.5ex}
\begin{equation}
  x_{s,k+1}\htts = \left\{
  \renewcommand{\arraystretch}{1.4}
  \begin{array}{ll}\vspace*{-4.125ex}\\
    \htbs
    a_{1,2} a_{2,1}^{-1} x_{2,1}
      &\,\mathrm{if}\quad\httbs k = 1\hts, \\
    \htbs
    -a_{k-1,k} a_{k+1,k}^{-1}\hts x_{s,k-1}
      &\,\mathrm{if}\quad\httbs k > 1\hts .
  \end{array}\right.
\label{KWextUpperQinf}
\end{equation}
In addition, if $|\hat{q}_{k}| = \infty$, then for $1 < k\leq n-1$\vspace{0.5ex}
\begin{equation}
  x_{k,k+1}\htts = \left\{
  \renewcommand{\arraystretch}{1.4}
  \begin{array}{ll}\vspace*{-4.33ex}\\
    \htbs
    \hat{r}_{k}^{\htts-1} x_{k-1,k+1}
      &\,\mathrm{if}\quad\httbs \hat{r}_{k}\neq 0 \hts, \\
    \htbs
    -a_{k-1,k-2} a_{k-1,k}^{-1}\hts x_{k-2,k+1}
      &\,\mathrm{if}\quad\httbs \hat{r}_{k} = 0,\, a_{k-1,k}\neq 0 \hts, \\
    \htbs
    -a_{k,k-1} a_{k+1,k}^{-1}\hts q_{k}^{-1}\hts x_{k-1,k-1}
      &\,\mathrm{if}\quad\httbs \hat{r}_{k} = 0,\, a_{k-1,k} = 0,\, q_{k}\neq 0 \hts, \\
    \htbs
    a_{k-1,k-2} a_{k+1,k}^{-1}\hts x_{k-2,k-1}
      &\,\mathrm{if}\quad\httbs \hat{r}_{k} = 0,\, a_{k-1,k} = 0,\, q_{k} = 0,\, a_{k,k-1}\neq 0 \hts, \\
    \htbs
    a_{k,k+1} a_{k+1,k}^{-1}\hts x_{k+1,k}
      &\,\mathrm{otherwise}\hts .
  \end{array}\right.
\label{KWextUpperMissing}
\vspace{-0.5ex}
\end{equation}
\end{theorem}
\begin{proof}
The formulae (\ref{KWextLower})--(\ref{KWextUpperMissing}) result from the equations
$AX\httbs=\htbs I\htbs=\htbs X\htbs A$,
the relations (\ref{LewisSymm}) and (\ref{RatiosSymm}), and from Theorem \ref{TwZeroBlocks}.
The complete proof is not very difficult, but is quite long. Therefore, we shall
justify only the one before last formula in (\ref{KWextLowerMissing}),
which --- we think --- is the most difficult one to prove.

From Theorem \ref{TwZeroBlocks}, we conclude that
if $|q_k|=\infty$ and $\hat{q}_k=a_{k+1,k}=0$, then $x_{k+1,k+1}=0$
and $x_{ij}=0$ for $k+1\leq i\leq n$, $1\leq j\leq n$. This implies that the closest non-zero
element to $x_{k,k-1}$ in the lower triangle is $x_{k+2,k+1}$, but there is no ''multiplicative
path'' between these two elements that would lead through the lower triangle only. However, in this case,
we also (by Theorem \ref{TwZeroBlocks}) know that $x_{k,k+2} = 0$, $x_{k-1,k+1} = 0$, and that $x_{k+1,k+2}\neq 0$,
$x_{k-1,k+2}\neq 0$, and $x_{k-1,k}\neq 0$, as $A$ is non-singular and $a_{k,k+1}\neq 0$.
Consequently, from the equations $AX=I$ and $XA=I$, and from (\ref{LewisSymm}),
we have\vspace{-1.75ex}
\begin{equation*}
\renewcommand{\arraystretch}{2.5}
\begin{array}{r@{\,\,=\,\,}l}
  \ds x_{k,k-1}
 &\ds \frac{a_{k,k-1}}{a_{k-1,k}}\htts x_{k-1,k}
    = \frac{a_{k,k-1}}{a_{k-1,k}}\htts \left( -\frac{a_{k+2,k+1}}{a_{k,k+1}}\htts x_{k-1,k+2} \right)\\
 &\ds - \frac{a_{k,k-1}}{a_{k-1,k}}\htts \frac{a_{k+2,k+1}}{a_{k,k+1}}\htts
    \left( - \frac{a_{k,k+1}}{a_{k,k-1}}\htts x_{k+1,k+2} \right)
    = \frac{a_{k+2,k+1}}{a_{k-1,k}}\htts \left( \frac{a_{k+1,k+2}}{a_{k+2,k+1}}\htts x_{k+2,k+1} \right)\\
 &\ds \frac{a_{k+2,k+1}}{a_{k-1,k}}\htts x_{k+2,k+1}\hts .
\end{array}
\vspace{0.25ex}
\end{equation*}

All other equations of Theorem \ref{TwKWext} can
be proved in an analogous way. What still may need a little more explanation
is that, e.g., the second and fourth equations of (\ref{KWextLowerMissing})
refer to the element $a_{k+1,k+2}$ which does not exist if $k=n-1$. However,
it can be proved (using. e.g., Theorem \ref{TwZeroBlocks}) that if $A$ is a
non-singular matrix, then the conditions required by these two
equations can be satisfied only for $k < n-1$.
\end{proof}

The formulae (\ref{KWextLower})--(\ref{KWextUpperMissing}) may be considered as a detailed
description of the new algorithm which can invert any non-singular tridiagonal matrix $A$
if we add one initial step: $\hts x_{n,n} = (a_{n-1,n}q_n + a_{n,n})^{-1}$. The algorithm is
not as elegant as, e.g, pivoting in the case of Gaussian elimination, but has a very important
feature: has the same complexity as its basic version, i.e.\ $n^2+O(n)$. What is even more
important, the new extended algorithm has the same numerical properties as the
one given by (\ref{KWbasic}). Some doubts may be related to the last formulae
of (\ref{KWextLowerMissing}) and (\ref{KWextDiag}), where we, in fact, use another starting
element. However, these two cases correspond to the situation, where $A$ is
a block diagonal matrix, and so
\begin{equation}
A^{-1}\hs =\hs\hs \begin{pmatrix}
F & 0\hs \\
0 & G\hs
\end{pmatrix}^{-1} =\hs\hs\hs
\begin{pmatrix}
F^{-1} & 0 \\
0 & G^{-1}
\end{pmatrix} .
\label{BlockDiagInv}
\end{equation}
In this particular case only, the matrices $F$ and $G$ can be inverted independently, as
they are in no way related in the equations $AX=I=X\htbs A$. Note that with the proposed
scheme, there is no need for special treatment of cases analogous to (\ref{BlockDiagInv}),
as the initial element for the inverse matrix $F^{-1}$ is computed --- as one could
say --- on the way.

The only limitation of the proposed new method for inverting general tridiagonal
matrices $A\in\mathbb{R}^{n\times n}$ is that it requires $n\geq 2$. Obviously,
if $n=1$, then $X = [a_{1,1}^{-1}\hts]$.


\end{document}